\documentclass[notitlepage]{scrartcl}
\usepackage[shortlabels]{enumitem}
\usepackage{amsmath}
\usepackage{tikz}
\usepackage{amssymb}
\usepackage{amsthm}
\usepackage{abstract}
\usepackage{xcolor}
\usepackage{comment}
\usepackage{mathtools}
\usepackage{graphicx}
\usepackage{caption}
\usepackage{subcaption}
\usepackage{float}
\usepackage{hyperref}
\usepackage{bm}
\usepackage{enumitem}
\usepackage{todonotes}
\makeatletter
\def\namedlabel#1#2{\begingroup
   \def\@currentlabel{#2}%
   \label{#1}\endgroup
}
\makeatother
\usepackage{dsfont}
\hypersetup{
    colorlinks=true,
    linkcolor=blue,
    filecolor=magenta,      
    urlcolor=cyan
    }
\makeatletter
\newcommand*{\barfix}[2][.175ex]{%
  \mathpalette{\@barfix{#1}}{#2}%
}
\newcommand*{\@barfix}[3]{%
  \vbox{%
    \kern#1\relax
    \hbox{$#2#3\m@th$}%
  }%
}
\makeatother

\usepackage[capitalise]{cleveref}
\newtheorem{thm}{Theorem}
\newtheorem{theorem}{Theorem}[section]

\newtheorem{lemma}[theorem]{Lemma}

\theoremstyle{definition}

\newcommand{\footremember}[2]{%
    \footnote{#2}
    \newcounter{#1}
    \setcounter{#1}{\value{footnote}}%
}
\newcommand{\footrecall}[1]{%
    \footnotemark[\value{#1}]%
}

\usepackage[margin=1in]{geometry}

\newcommand{\cC}{\mathcal{C}}

\newcommand{\cL}{\mathcal{L}}

\newcommand{\cQ}{\mathcal{Q}}

\newcommand{\e}{\mathrm{e}}
\newcommand{\eps}{\varepsilon}

\title{\vspace{-1.5cm}Cycle lengths in the percolated hypercube}
\author{%
Michael Anastos \footremember{alley1}{\scriptsize{Institute of Science and Technology Austria (ISTA), Klosterneurburg 3400, Austria. Email:  michael.anastos@ist.ac.at.}}%
\and Sahar Diskin \footremember{alley2}{\scriptsize{School of Mathematical Sciences, Tel Aviv University, Tel Aviv 6997801, Israel. Emails: sahardiskin@mail.tau.ac.il, krivelev@tauex.tau.ac.il.}}%
\and Joshua Erde \footremember{alley3}{\scriptsize{Institute of Discrete Mathematics, Graz University of Technology, Steyrergasse 30, 8010 Graz, Austria. Emails: erde@math.tugraz.at, kang@math.tugraz.at.}}%
\and Mihyun Kang \footrecall{alley3}%
\and Michael Krivelevich \footrecall{alley2}%
\and Lyuben Lichev \footremember{alley4}{\scriptsize{Institute of Statistics and Mathematical Methods in Economics, Technical University of Vienna, A-1040 Vienna, Austria. Email: lyuben.lichev@tuwien.ac.at.}}%
}
\date{\vspace{-1.5cm}}

\begin{document}

\maketitle

\begin{abstract}
Let $Q^d_p$ be the random subgraph of the $d$-dimensional binary hypercube obtained after edge-percolation with probability $p$. 
It was shown recently by the authors that, for every $\eps > 0$, there is some $c = c(\eps)>0$ such that, if $pd\ge c$, then typically $Q^d_p$ contains a cycle of length at least $(1-\eps)2^d$. 
We strengthen this result to show that, under the same assumptions, typically $Q^d_p$ contains cycles of all even lengths between $4$ and $(1-\eps)2^d$.
\end{abstract}

\section{Introduction}

The study of \emph{Hamiltonicity} of graphs is a central topic in probabilistic combinatorics. Whilst in general the problem of determining whether a graph is Hamiltonian is known to be computationally hard \cite{K72}, classic structural criteria for the existence of Hamilton cycles in dense graphs due to Dirac~\cite{Dir52}, Ore~\cite{Ore60}, Chv\'atal~\cite{Chv72}, and Bondy and Chv\'atal~\cite{BC76} have been known for the last half a century.
Yet, understanding the property in satisfactory generality is highly non-trivial and has given birth to some fascinating mathematical advances.
P\'osa~\cite{Pos76} pioneered the study of Hamiltonicity in random graphs, developing the \emph{rotation-extension method} and using it to show that $G(n,p)$, the random graph on $n$ vertices where each edge appears independently with probability $p$, is typically Hamiltonian when $np = C\log n$ for a suitably large constant $C$.
Variations of this method have proved instrumental in a number of works on the subject.
P\'osa's result was further improved by Korshunov~\cite{Kor77}, Bollob\'as~\cite{Bol84}, and Koml\'os and Szemer\'edi~\cite{KS83}. 
In the sparser regime, Ajtai, Koml\'os, and Szemer\'edi~\cite{AKS81a} showed that, for any constant $c>1$, 
$G(n,c/n)$ typically contains a cycle of length $(1-f(c))n$ where $f:(1,\infty)\to (0,1)$ is some function with $\lim_{x\to\infty} f(x)=0$ (see also the independently obtained result of Fernandez de la Vega \cite{Fdlv79}).

A natural research direction generalising the advances on Hamiltonicity is the study of \emph{pancyclicity}. Denoting by $\cL(H)$ the set of cycle lengths in a given graph $H$, also known as the \emph{cycle spectrum} of $H$, we say that $H$ is pancyclic if $\cL(H)=\{3,\ldots, |V(H)|\}$ that is, if $H$ contains cycles of all possible lengths. 
In the setting of $G(n,p)$, Cooper and Frieze~\cite{CF90} established a hitting time result for pancyclicity in the random graph process, showing that minimum degree two is typically both a necessary and a sufficient condition. 
In the sparser setting, Alon, the fifth author, and Lubetzky \cite{AKL22} showed that, for sufficiently large $c>0$ and for any $\ell=\ell(n)$ tending arbitrary slowly to infinity, typically $\cL(G(n,c/n)) \supseteq \{\ell,\ell+1,\ldots,(1-o(1))L_{max}\}$, where $L_{max}$ is the length of a longest cycle in $G(n,c/n)$. 
This was further improved by the first author~\cite{A23} who calculated the probability that, for $c>0$ sufficiently large, $G(n,c/n)$ is \textit{weakly pancyclic}, that is, contains cycles of all lengths between $3$ and the length of a longest cycle. 

The binomial random graph $G(n,p)$ is perhaps the simplest example of a model of \textit{$p$-bond percolation} where every edge of a host graph $G$ is retained independently with probability $p$, thus producing a random subgraph $G_p$. Thus, $G(n,p)$ can be seen as $p$-bond percolation on a $K_n$, the complete graph on $n$ vertices. Studying foundational topics like Hamiltonicity and pancyclicity in other models involving bond percolation is a natural and inviting question.

Perhaps the second most studied host graph after the complete graph is the high-dimensional binary hypercube.
For an integer $d\ge 1$, the \emph{$d$-dimensional hypercube} $Q^d$ is the graph with vertex set $\{0,1\}^d$ where every two vertices differing in a single coordinate form an edge.
The non-trivial lattice-like geometry of the hypercube $Q^d$ has proved a significant obstacle for adapting the methods used in the analysis of $G(n,p)$ to the percolated hypercube $Q^d_p$. 
For example, a phase transition for the existence of a giant component in percolated hypercubes was established by Ajtai, Koml\'os and Szemer\'edi~\cite{AKS82}, see also Bollob\'as, Kohayakawa and {\L}uczak~\cite{BKL92}, more than two decades after the original work of Erd\H{o}s and R\'enyi on $G(n,p)$~\cite{ER60}. 
Furthermore, thresholds for the existence of long cycles and of Hamiltonian cycles in $Q^d_p$ were resolved only recently~\cite{ADEKKL25,CEGKO24}. In a breakthrough, Condon, Espuny D\'{\i}az, Gir\~{a}o, K\"{u}hn and Osthus~\cite{CEGKO24} established a sharp threshold for Hamiltonicity in $Q^d_p$ at $p=1/2$. In fact, they even showed a stronger \emph{hitting time} result, and with minor modifcations of their approach, one can obtain a hitting time result for \textit{even pancyclicity} as well.
In the (much) sparser regime, the authors~\cite{ADEKKL25} showed that, for large enough $c>0$, \textbf{whp}\footnote{With high probability, that is, with probability tending to one as $d$ tends to infinity.} $Q^d_{c/d}$ contains a cycle of length $(1-f(c))2^d$ where $f:(1,\infty)\to (0,1)$ is some function with $\lim_{x\to\infty} f(x)=0$.

In this paper, we consider the cycle spectrum of $Q^d_p$. Note that, since the hypercube is bipartite, 
$Q^d$ (and $Q^d_p$) does not contain any cycles of odd length. Our main result is an analogue of the result of~\cite{AKL22} in the setting of percolated hypercubes.

\begin{thm}\label{thm:main1}
For every $\eps\in (0,1)$, there exists a constant $c(\eps)>0$ such that, for all $c\geq c(\eps)$, with $p = p(d) = c/d$, \textbf{whp} $Q^d_p$ simultaneously contains all cycles of even length between $4$ and $(1-\eps) 2^d$. 
\end{thm}
We note that the dependency of $c$ on $\eps$ which we obtain is inverse polynomial. Let us further note that Bollob\'as (see \cite{bollobas2001random}, \S 4.1), and separately Karo\'nski and Ruci\'nski~\cite{karonckirusinci} proved that for fixed $\ell$ the number of cycles of length $\ell$ in $G(n,c/n)$ converges to Poisson($c^\ell/(2\ell)$), hence $G(n,c/n)$ does not contain a cycle of length $\ell$ with probability uniformly bounded away from~$0$. On the other hand, a second moment argument shows that typically $Q_{c/d}^d$ does contain small cycles of a given even length.

Our work leaves some natural open problems. 
It would be very interesting to determine if, analogous to the case of $G(n,p)$, when $pd$ is sufficiently large, $Q^d_p$ is typically \emph{weakly even-pancyclic}, that is, it contains cycles of all even lengths up to $L_{max}$. 
Settling a well-known conjecture of the fifth author and Sudakov \cite{KS03}, a recent work by Dragani\'c, Montgomery, Munh\'a Correia, Pokrovskiy and Sudakov \cite{DMCPS24} showed that \textit{$C$-expanders} are Hamiltonian. Results for cycle lengths have also been shown for \emph{pseudorandom graphs}~\cite{FK21,HMT09} (which are, in particular, $C$-expanders). It would be interesting to know if these results are \emph{robust} enough to still hold after percolation, as is known to be the case for the existence of long cycles, see~\cite{DK24}. 
The fifth author, Lee and Sudakov \cite{KLS10} studied the \textit{local resilience} of pancyclicity in $G(n,p)$. Lee and Samotij \cite{LS12} considered the \textit{global resilience} of pancyclicity in $G(n,p)$. Similar questions about resilience, both local and global, can now be asked in the setting of the percolated hypercube.

\paragraph{Structure of the paper.}
In Section \ref{s: prelim}, we set out notation and present some preliminary results.
Section \ref{sec:proof} is dedicated to the proof of Theorem~\ref{thm:main1}.

\section{Preliminaries}\label{s: prelim}
\subsection{Notation}
For a positive integer $n$, we write $[n] = \{1,\ldots,n\}$.
Rounding notation is systematically omitted for better readability whenever it does not affect the validity of our arguments. All logarithms are with respect to the natural basis $\e$. We denote by $\mathbb{N}$ the set of positive integers, and by $2\mathbb{N}$ the set of \textit{even} positive integers. 

Given a graph $G$, we denote its vertex set by $V(G)$ and its edge set by $E(G)$. For a set $A\subseteq V(G)$, we denote by $N_G(A)$ the external neighbourhood of $A$ in $G$, that is, the neighbours of $A$ in $V(G)\setminus A$. 
Given $A,B\subseteq V(G)$ with $A\cap B=\varnothing$, we denote by $e_G(A,B)$ the number of edges in $G$ with one endpoint in $A$ and the other endpoint in $B$. 
If the graph $G$ is clear from context, we often omit the subscript. Furthermore, we denote by $G[A]$ the subgraph of $G$ induced by $A$. Given $v\in V(G)$ and $r\in \mathbb{N}$, we denote by $B(v,r)$ the ball of radius $r$ centred at $v$, that is, the set of vertices at distance at most $r$ from $v$.

In the hypercube $Q^d$, the \emph{$i$-th layer} consists of the vertices of $Q^d$ with exactly $i$ coordinates equal to 1. A \textit{monotone path} in the hypercube is a path containing at most one vertex on each layer of that hypercube. A monotone path is \emph{maximal} when its length (that is, number of edges) coincides with the dimension of the host hypercube.

Given $p\in [0,1]$, we form $Q^d_p$ by retaining every edge of $Q^d$ independently and with probability $p$. Given $q,p\in [0,1]$, we define the \emph{mixed-percolated hypercube} $Q^d_p(q)$ as the graph $Q^d_p[V_q]$ where $V_q \subseteq V(Q^d)$ is a random set obtained by retaining every vertex independently with probability~$q$.

Throughout the paper, when considering subgraphs $H\subseteq Q^d$, we slightly abuse notation and write $H_p$ for the random subgraph $H \cap Q^d_p$. This naturally couples the graphs $H_p$ and $H'_p$ for different $H,H' \subseteq Q^d$.

\subsection{Auxiliary results}
We will use the following classical Chernoff bound (see, for example, \cite[Theorem A.1.12]{AS16}).
\begin{lemma}\label{l: chernoff}
For any binomial random variable $X$ and $a\in [0,\mathbb EX]$,
\begin{align*}
    \mathbb{P}\left[|X-\mathbb EX|\ge a\right]\le 2\exp\left(-\frac{a^2}{3\mathbb EX}\right).
\end{align*}
\end{lemma}
We will often use Lemma \ref{l:  chernoff} when estimating a one-sided tail bound of $X\sim \mathrm{Bin}(n,p)$, noting that
\begin{align*}
    \mathbb{P}\left[X\le a\right]\le \mathbb{P}\left[|X-\mathbb EX|\ge a-\mathbb EX\right] \quad \text{and}\quad \mathbb{P}[X\ge a]\le \mathbb{P}[|X-\mathbb EX|\ge a-\mathbb EX].
\end{align*}

Next, we state a simplified version of the main theorem in~\cite{ADEK24}, estimating the probability of existence of maximal monotone paths in supercritical percolation on high-dimensional hypercubes.
\begin{theorem}[see Theorem~1 in~\cite{ADEK24}]\label{thm:monotone}
For all sufficiently large integers $D$, the following holds. 
Let $\rho = \rho(D) > \e/D$. Then,~$Q^D_\rho$ contains a maximal monotone path with probability at least~$1/2$.
\end{theorem}

We will also utilise the following estimate on the probability of the existence of maximal monotone paths in percolated hypercubes.
\begin{lemma}\label{lem:short}
Let $D\geq 1$ be an integer and $\rho=\rho(D) \le 1/D$. Then, $Q^D_\rho$ contains a maximal monotone path with probability at least $(\rho D/(2\e))^D$.
\end{lemma}
\begin{proof}
We analyse the following random greedy algorithm. Denote by $v_0$ the all-zero vertex of $Q^D$, and set $\mathbf{i} = 0$. At each step $0\le \mathbf{i}\le D-1$, consider the vertex $v_{\mathbf{i}}$ and expose the edges from $v_{\mathbf{i}}$ to the $(\mathbf{i}+1)$-st layer in the hypercube $Q^D_\rho$. 
If no such edge exists, the algorithm terminates. Otherwise, fix the first such edge (according to an arbitrary order), call its other endpoint $v_{\mathbf{i}+1}$ and increment $\mathbf{i}$ by 1.

Note that, in this algorithm, no edge is exposed more than once. Moreover, since every vertex on the $i$-th layer contains $D-i$ edges towards the $(i+1)$-st layer for all $i\in [0,D-1]$, conditionally on the event that the algorithm does not terminate before reaching layer $i$, the probability that it reaches layer $i+1$ is $\mathbb P[\mathrm{Bin}(D-i,\rho)\ge 1] = 1 - (1-\rho)^{D-i}$.
Hence, the probability that the algorithm reaches layer $D$, which implies the existence of a maximal monotone path, is
\[\prod_{i=0}^{D-1}(1 - (1-\rho)^{D-i}) = \prod_{j=1}^{D}(1 - (1-\rho)^j)\ge \prod_{j=1}^{D}(1 - \e^{-j \rho})\ge D! \bigg(\frac{\rho}{2}\bigg)^D \ge \bigg(\frac{\rho D}{2 \e}\bigg)^D,\]
where the first inequality above uses that $1-t\leq \e^{-t}$ for every $t\ge 0$, the second inequality holds since $1-\e^{-t}\ge t/2$ for every $t\in [0,1]$, and the third inequality uses that $D!\ge (D/\e)^D$ for every integer $D\ge 1$.
\end{proof}

Our proof also critically relies on the aforementioned recent result from~\cite{ADEKKL25} on the existence of a nearly spanning cycle in the (mixed) percolated hypercube.
\begin{theorem}[Theorem 1 and Remark 1 in \cite{ADEKKL25}]\label{thm:long cycle}
For every fixed $q\in (0,1]$ and $\eps \in (0,1)$, there exists a constant $c_0 = c_0(q,\eps) > 0$ such that, for every $c \ge c_0$, \textbf{whp} a longest cycle in $Q^d_{c/d}(q)$ has length at least $(1-\eps) q\cdot 2^d$.
\end{theorem}

We require some results on the component structure and typical properties of the giant component in bond percolation on the hypercube.
We recall that the \emph{order} of a connected graph is the number of vertices in it.

\begin{theorem}\label{th: bond percolation}
Let $c>1$ and $p=c/d$. Then, there exist constants $C = C(c)>0$ and $y=y(c)>0$ such that the following properties hold \textbf{whp}.
\begin{enumerate}[\upshape{(\alph*)}]
    \item\label{i: bond giant} There exists a unique giant component in $Q^d_p$ whose order is at least $y 2^d$. All other components have order at most $Cd$.
    \item\label{i: bond diameter} The diameter of the giant component is at most $Cd(\log d)^2$.
\end{enumerate}
\end{theorem}

\Cref{th: bond percolation}\ref{i: bond giant} follows from the classical results of Ajtai, Koml\'os and Szemer\'edi~\cite{AKS82} and Bollob\'as, Kohayakawa and {\L}uczak~\cite{BKL92} (see \cite{K23} for a simple and self-contained proof). \Cref{th: bond percolation}\ref{i: bond diameter} follows from \cite[Theorem 6(a)]{DEKK24} (note that, whilst \cite[Theorem 6(a)]{DEKK24} is stated for $c$ sufficiently close to $1$, as discussed in \cite[Page 747, first paragraph]{DEKK24}, the results naturally extend to any constant $c>1$ when allowing the constant $C$ to depend on $c$).

We also require some variant of the above results (and other typical properties) in the setting of mixed percolation. 
\begin{theorem}\label{th: mixed percolation}
Let $q\in (0,1]$. Then, there exists a constant $c_0 = c_0(q)$ such that, for all $c\ge c_0$, there are $y = y(c,q) > 0$ and $C = C(c,q) > 0$ such that the following properties hold~\textbf{whp}. 
\begin{enumerate}[\upshape{(\alph*)}]
    \item\label{i: mixed giant} There exists a unique giant component in $Q^d_{c/d}(q)$ whose order is at least $y2^d$. All other components have order at most $Cd$.
    \item\label{i: mixed diameter} The diameter of the giant component is at most $2^d/d^{8}$.
\end{enumerate}
\end{theorem}
Let us stress here that Theorem~\ref{th: mixed percolation}\ref{i: mixed diameter} is far from being tight. However, this crude estimate will suffice for our needs. 
The proof of Theorem~\ref{th: mixed percolation}\ref{i: mixed giant} follows the proof of~\cite[Theorem~2]{K23} with very straightforward modifications in Lemmas 4--8 therein due to the vertex percolation.

We turn our attention to the proof of Theorem~\ref{th: mixed percolation}\ref{i: mixed diameter}, which requires some preparation. The following lemma provides an upper bound on the number of $k$-vertex trees in a graph
$G$ of maximum degree $d$. It is an immediate consequence of \cite[Lemma 2]{BFM98}.
\begin{lemma}\label{l: trees}
Let $G$ be a graph of maximum degree at most $d$,  $v\in V(G)$ and $k\ge 1$ be an integer. Denote by $t_{k}(G,v)$ the number of trees on $k$ vertices in $G$ rooted~at~$v$. Then,
\begin{align*}
     t_k(G,v)\le (\e d)^{k-1}.
\end{align*}
\end{lemma}
We will also use a simplified version from~\cite{Gal03} of the well-known Harper's vertex-isoperimetric inequality~\cite{H66}.

\begin{lemma}[Lemma~6.2 in~\cite{Gal03}]\label{l: Harper}
There exists a constant $c_1 > 0$ such that, for every set $S\subseteq V(Q^d)$ of size $|S|\le d^{10}$, one has $|N(S)|\ge c_1d|S|$.
\end{lemma}

With these two lemmas at hand, we can show that sets $S\subseteq V(Q^d)$ of size $k\in [d, d^{10}]$ typically have a constant vertex-expansion in $Q^d_p(q)$, assuming they span a connected subgraph of $Q^d_p(q)$.
\begin{lemma}\label{l: expansion}
Fix $q\in (0,1]$. Then, there exist positive constants $c_2=c_2(q)$ and $a=a(c_2,q)$ such that, for all $c\geq c_2$ and $p=c/d$, the following holds \textbf{whp}. For every $S\subseteq V(Q^d_p(q))$ of size $k\in [d, d^{10}]$ such that $Q^d_p(q)[S]$ is connected, we have that $|N_{Q^d_p(q)}(S)|\ge ak$.
\end{lemma}
\begin{proof}
Let $c_1$ be the constant whose existence is guaranteed by Lemma~\ref{l: Harper}. 

On the event that there exists a set $S\subseteq V(Q^d_p(q))$ of size $k\in [d, d^{10}]$ such that $Q^d_p(q)[S]$ is connected and $|N_{Q^d_p(q)}(S)|< ak$, the following holds. There exist $k\in [d, d^{10}]$, a vertex $v\in V(Q^d)$ and a tree $T$ on $k$ vertices in $Q^d$ such that each of the following holds:
\begin{itemize}
    \item $v\in V(T)$;
    \item $T$ is a subgraph of $Q^d_p(q)$;
    \item out of the $|N_{Q^d}(V(T))|$ neighbours of $V(T)$ in $Q^d$, at most $ak$ belong to $V(Q^d_p(q))$ and are connected to $V(T)$ in $Q^d_p(q)$.
\end{itemize}
There are $2^d$ choices for $v$ and, thereafter, by Lemma~\ref{l: trees}, at most $(\e d)^{k-1}$ many choices for a tree $T\subseteq Q^d$ on $k$ vertices that contains $v$. Each such tree $T$ is present in $Q^d_p(q)$ with probability $q^kp^{k-1}$. 
Conditionally on $T \subseteq Q^d_p(q)$, each vertex in $N_{Q^d}(V(T))$  belongs to $V(Q^d_p(q))$ and is connected to $T$ in $Q^d_p(q)$ independently with probability at least $pq$. Moreover, Lemma~\ref{l: Harper} and the fact that $k\le d^{10}$ imply that $|N_{Q^d}(V(T))|\ge c_1 kd$. Thus, 
\begin{align*}
\mathbb{P}\left[|N_{Q^d_p(q)}(V(T))|<ak\right]\le \mathbb{P}\left[\mathrm{Bin}(c_1kd,pq)\le c_1kdpq/2\right],
\end{align*}
where we assumed that $a\le c_1c_2q/2\le c_1cq/2=c_1dpq/2$. Thus, by the union bound, the probability that there exists a set $S\subseteq V(Q^d_p(q))$ of size $k\in [d, d^{10}]$ such that $Q^d_p(q)[S]$ is connected and $|N_{Q^d_p(q)}(S)|< ak$ is at most
\begin{align*}
\sum_{k=d}^{d^{10}} &2^d (\e d)^{k-1}p^{k-1}q^k  \mathbb P[\mathrm{Bin}(c_1kd,qp)\leq c_1kdqp/2] \leq \sum_{k=d}^{d^{10}} 2^d(\e dpq)^{k} 2\exp\left\{-\frac{ c_1 kdqp }{12}\right\}
\\& = \sum_{k=d}^{d^{10}} 2^{d+1} \left(\e cq\cdot \exp\left\{- \frac{c_1cq}{12} \right\} \right)^k \leq \sum_{k=d}^{d^{10}} 2^{d+1} \big(2^{-2}\big)^k \leq d^{10} 2^{-d+1} =  o(1),
\end{align*}
where the first inequality follows from Chernoff's bound (Lemma~\ref{l: chernoff}) and the second inequality uses that $c\geq c_2(q)$ is suitably large with respect to $q$. This finishes the proof.
\end{proof}

We are now ready to prove Theorem \ref{th: mixed percolation}\ref{i: mixed diameter}.
\begin{proof}[Proof of Theorem \ref{th: mixed percolation}\ref{i: mixed diameter}]
Fix $p = c/d$. By Lemma \ref{l: expansion}, \textbf{whp} every connected set in $Q^d_p(q)$ of size between $d$ and $d^{10}$ has vertex-expansion by a factor of at least $a$ for some constant $a>0$. We assume in the sequel that this property holds. 

We may assume that the giant component of $Q^d_p(q)$, which we denote by $L_1$, satisfies that $|V(L_1)|\ge \frac{2^d}{d^{10}}\ge d$. Thus, for every $v\in V(L_1)$, we have that $|B(v,d)|\ge d$. Hence, by our vertex-expansion assumption, for every $v\in V(L_1)$ we have that $|B(v,2d)|\ge d^{10}$.

Fix two vertices $u,v$ in $L_1$ and let $P=\{v_0,\ldots, v_t\}$ be a shortest path (in $L_1$) between them, where $u=v_0$ and $v=v_t$. Let $x_1,\ldots, x_{t/5d}$ be a set of $\frac{t}{5d}$ vertices along $P$, such that the distance between $x_i$ and $x_j$ is at least $5d$, for any $i \neq j$. Hence, for every $i\neq j$, we have that $B(x_i,2d)\cap B(x_j,2d)=\varnothing$. Thus, $\sum_{i=1}^{t/5d}|B(x_i,2d)|\le |L_1|\le 2^d$. On the other hand, by the above, $\sum_{i=1}^{t/5d}|B(x_i,2d)|\ge \frac{t}{5d}\cdot d^{10}$. Therefore, $t\le \frac{5\cdot2^d}{d^9}<\frac{2^d}{d^8}$. This completes the proof.  
\end{proof}
 
\section{\texorpdfstring{Proof of \Cref{thm:main1}}{}}\label{sec:proof}
The proof proceeds in four steps, which separately guarantee that, typically, cycles in $Q^d_p$ with any length in the following sets $I_1,\ldots,I_4$ exist:
\begin{itemize}
    \item (Very short cycles) $I_1 = [4,d/5]\cap 2\mathbb{N}$;
    \item (Short cycles) $I_2 = [d/5,d^{10}]\cap 2\mathbb{N}$;
    \item (Medium cycles)  $I_3 = [d^{10}, 2^{d-4}]\cap 2\mathbb{N}$;
    \item (Long cycles) $I_4 = [2^{d-4},(1-\eps)2^d]\cap 2\mathbb{N}$.
\end{itemize}
Since the union $I_1\cup I_2\cup I_3\cup I_4 \supseteq [4,(1-\eps)2^d]\cap 2\mathbb{N}$, this is enough to derive \Cref{thm:main1}.

\paragraph{Finding very short cycles.}
First, we show the typical existence of all cycles with lengths in $I_1$. In fact, we will show a slightly stronger result which will be useful in the subsequent cases. 

\begin{lemma}\label{l: very short cycles}
Let $D$ be a sufficiently large integer and $\rho =\rho(D) \geq 10/D$. Then, with probability~at~least $1-\exp\left\{-1.1^D\right\}$, $Q^D_\rho$ contains simultaneously all cycles of even length in the interval $[4,D/5]$.
\end{lemma}
\begin{proof}
Since the property of containing cycles of every even length in $[4,D/5]$ is monotone with respect to $\rho$, we may (and do) assume that $\rho=10/D$. Fix an integer $\ell = \ell(D)\in [1,D/10-1]$. Then, by \Cref{lem:short}, the probability that $Q^{\ell}_\rho$ contains a maximal monotone path is bounded from below by $\eta = \eta(\ell) := (\ell \rho/(2 \e))^{\ell}$.

Divide the hypercube $Q^D$ into a family $\cQ$ of $2^{D-\ell-1}$ disjoint copies of $Q^{\ell+1}$. For every copy $Q\in \cQ$ of $Q^{\ell+1}$, we split $Q$ into disjoint copies $Q_0$ and $Q_1$ of $Q^{\ell}$, which are joined by a matching. 
Note that if $(Q_0)_\rho$ and $(Q_1)_\rho$ simultaneously contain maximal monotone paths $P_0,P_1$ whose starting and ending points are also connected in $Q^D_\rho$, then $Q^D_\rho$ contains a cycle of length $2\ell+2$. Note further that these events are independent for every two $Q,Q'\in \cQ$. Thus, the probability that $Q^D_\rho$ does not contain a cycle of length $2\ell+2$ is at most
\begin{align*}
    (1-\rho^2\eta^2)^{2^{D-\ell-1}}&\le \exp\left\{-\rho^2\eta^22^{D-\ell-1}\right\}=\exp\left\{-\left(\frac{10}{D}\right)^2\left(\frac{\ell \rho}{2\e}\right)^{2\ell}2^{D-\ell-1}\right\}\\
    &\le \exp\left\{-D^{-2}\left(\frac{10\ell}{3\e D}\right)^{2\ell}2^D\right\}\\
    &\le \exp\left\{-D^{-2} \cdot 1.3^D\right\}\le \exp\left\{-1.2^D\right\},
\end{align*}
where the third inequality comes from the fact that $1.3$ is a lower bound for the function $x\in (0,1/10]\mapsto 2(10x/(3\e))^{2x}$.
The union bound over the $D/10-1$ cycle lengths completes the proof.
\end{proof}

\begin{proof}[Proof of \Cref{thm:main1} for the lengths in $I_1$]
Follows immediately from Lemma \ref{l: very short cycles} with $D=d$, and assuming that $c\ge 10$.
\end{proof}

\paragraph{Finding short cycles.}

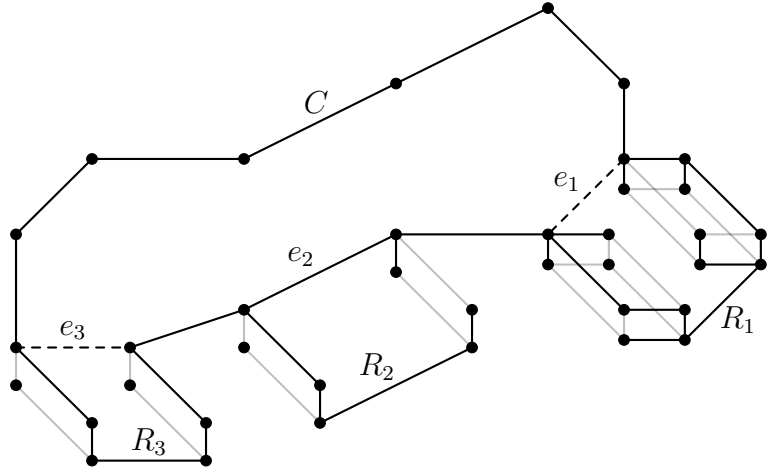
\begin{figure}
\centering
\begin{tikzpicture}[line cap=round,line join=round,x=1cm,y=1cm]
\clip(-11.5,-3.1) rectangle (7,3.1);
\draw [line width=0.8pt,dash pattern=on 3pt off 3pt] (-8,-1.5)-- (-6.5,-1.5); % e1 dashed
\draw [line width=0.8pt] (-6.5,-1.5)-- (-5,-1); % e1 continued
\draw [line width=0.8pt] (-8,-1.5)-- (-8,0);
\draw [line width=0.8pt] (-8,0)-- (-7,1);
\draw [line width=0.8pt] (-7,1)-- (-5,1);
\draw [line width=0.8pt] (-5,1)-- (-3,2);
\draw [line width=0.8pt] (-3,2)-- (-1,3);
\draw [line width=0.8pt] (-1,3)-- (0,2);
\draw [line width=0.8pt] (0,2)-- (0,1);
\draw [line width=0.8pt,dash pattern=on 3pt off 3pt] (0,1)-- (-1,0); % e3 dashed
\draw [line width=0.8pt] (-1,0)-- (-3,0);
\draw [line width=0.8pt] (-3,0)-- (-5,-1); % e2 now solid
% Grey structural attachments R1, R2, R3 unchanged
\draw [line width=0.8pt,color=black, opacity=0.25] (-8,-1.5)-- (-8,-2);
\draw [line width=0.8pt,color=black, opacity=0.25] (-8,-2)-- (-7,-3);
\draw [line width=0.8pt,color=black] (-7,-3)-- (-7,-2.5);
\draw [line width=0.8pt] (-7,-2.5)-- (-8,-1.5);
\draw [line width=0.8pt,color=black, opacity=0.25] (-6.5,-1.5)-- (-6.5,-2);
\draw [line width=0.8pt,color=black, opacity=0.25] (-6.5,-2)-- (-5.5,-3);
\draw [line width=0.8pt] (-5.5,-3)-- (-5.5,-2.5);
\draw [line width=0.8pt] (-5.5,-2.5)-- (-6.5,-1.5);
\draw [line width=0.8pt] (-7,-3)-- (-5.5,-3);
\draw [line width=0.8pt,color=black, opacity=0.25] (-5,-1)-- (-5,-1.5);
\draw [line width=0.8pt,color=black, opacity=0.25] (-5,-1.5)-- (-4,-2.5);
\draw [line width=0.8pt] (-4,-2.5)-- (-4,-2);
\draw [line width=0.8pt] (-4,-2)-- (-5,-1);
\draw [line width=0.8pt] (-3,0)-- (-3,-0.5);
\draw [line width=0.8pt,color=black, opacity=0.25] (-3,-0.5)-- (-2,-1.5);
\draw [line width=0.8pt] (-2,-1.5)-- (-2,-1);
\draw [line width=0.8pt,color=black, opacity=0.25] (-2,-1)-- (-3,0);
\draw [line width=0.8pt] (-4,-2.5)-- (-2,-1.5);
\draw [line width=0.8pt] (-1,0)-- (-1,-0.4);
\draw [line width=0.8pt,color=black, opacity=0.25] (-1,-0.4)-- (0,-1.4);
\draw [line width=0.8pt] (0,-1.4)-- (0.8,-1.4);
\draw [line width=0.8pt] (-1,0)-- (-0.2,0);
\draw [line width=0.8pt,color=black, opacity=0.25] (-0.2,0)-- (0.8,-1);
\draw [line width=0.8pt] (0.8,-1)-- (0.8,-1.4);
\draw [line width=0.8pt,color=black, opacity=0.25] (-0.2,0)-- (-0.2,-0.4);
\draw [line width=0.8pt,color=black, opacity=0.25] (-0.2,-0.4)-- (-1,-0.4);
\draw [line width=0.8pt,color=black, opacity=0.25] (-0.2,-0.4)-- (0.8,-1.4);
\draw [line width=0.8pt] (0,-1)-- (0.8,-1);
\draw [line width=0.8pt,color=black, opacity=0.25] (0,-1)-- (0,-1.4);
\draw [line width=0.8pt] (-1,0)-- (0,-1);
\draw [line width=0.8pt] (0,1)-- (0,0.6);
\draw [line width=0.8pt,color=black, opacity=0.25] (0,0.6)-- (0.8,0.6);
\draw [line width=0.8pt] (0.8,0.6)-- (0.8,1);
\draw [line width=0.8pt] (0.8,1)-- (0,1);
\draw [line width=0.8pt,color=black, opacity=0.25] (0,0.6)-- (1,-0.4);
\draw [line width=0.8pt] (1,-0.4)-- (1,0);
\draw [line width=0.8pt,color=black, opacity=0.25] (1,0)-- (1.8,0);
\draw [line width=0.8pt] (1.8,0)-- (1.8,-0.4);
\draw [line width=0.8pt] (1.8,-0.4)-- (1,-0.4);
\draw [line width=0.8pt,color=black, opacity=0.25] (0.8,0.6)-- (1.8,-0.4);
\draw [line width=0.8pt,color=black, opacity=0.25] (0,1)-- (1,0);
\draw [line width=0.8pt] (0.8,1)-- (1.8,0);
\draw [line width=0.8pt] (0.8,-1.4)-- (1.8,-0.4);
\begin{scriptsize}
\draw [fill=black] (-8,-1.5) circle (2pt);
\draw [fill=black] (-8,0) circle (2pt);
\draw [fill=black] (-7,1) circle (2pt);
\draw [fill=black] (-5,1) circle (2pt);
\draw [fill=black] (-3,2) circle (2pt);
\draw [fill=black] (-6.5,-1.5) circle (2pt);
\draw [fill=black] (-5,-1) circle (2pt);
\draw [fill=black] (-1,3) circle (2pt);
\draw [fill=black] (0,2) circle (2pt);
\draw [fill=black] (0,1) circle (2pt);
\draw [fill=black] (-1,0) circle (2pt);
\draw [fill=black] (-3,0) circle (2pt);
\draw [fill=black] (-8,-2) circle (2pt);
\draw [fill=black] (-7,-3) circle (2pt);
\draw [fill=black] (-7,-2.5) circle (2pt);
\draw [fill=black] (-6.5,-2) circle (2pt);
\draw [fill=black] (-5.5,-3) circle (2pt);
\draw [fill=black] (-5.5,-2.5) circle (2pt);
\draw [fill=black] (-5,-1.5) circle (2pt);
\draw [fill=black] (-4,-2.5) circle (2pt);
\draw [fill=black] (-4,-2) circle (2pt);
\draw [fill=black] (-3,-0.5) circle (2pt);
\draw [fill=black] (-2,-1.5) circle (2pt);
\draw [fill=black] (-2,-1) circle (2pt);
\draw [fill=black] (-1,-0.4) circle (2pt);
\draw [fill=black] (0,-1.4) circle (2pt);
\draw [fill=black] (0.8,-1.4) circle (2pt);
\draw [fill=black] (-0.2,0) circle (2pt);
\draw [fill=black] (0.8,-1) circle (2pt);
\draw [fill=black] (-0.2,-0.4) circle (2pt);
\draw [fill=black] (0,-1) circle (2pt);
\draw [fill=black] (0,0.6) circle (2pt);
\draw [fill=black] (0.8,0.6) circle (2pt);
\draw [fill=black] (0.8,1) circle (2pt);
\draw [fill=black] (1,-0.4) circle (2pt);
\draw [fill=black] (1,0) circle (2pt);
\draw [fill=black] (1.8,0) circle (2pt);
\draw [fill=black] (1.8,-0.4) circle (2pt);

\draw (-7.25,-1.27) node {\large{$e_3$}};
\draw (-4.25,-0.3) node {\large{$e_2$}};
\draw (-0.75,0.7) node {\large{$e_1$}};
\draw (-6.25,-2.75) node {\large{$R_3$}};
\draw (-3.25,-1.75) node {\large{$R_2$}};
\draw (1.5,-1.13) node {\large{$R_1$}};
\draw (-4.05,1.75) node {\large{$C$}};
\end{scriptsize}
\end{tikzpicture}
\caption{Illustration of the proof of \Cref{l: inductive step}. Here, $k_1 = 3$ and $k_2 = 2$. Edges in the cubes $R_1$, $R_2$ and $R_3$ missing from $Q^d_p$ are depicted in light grey.
We fail to extend the cycle $\cC$ in $R_2$: indeed, the right part of $R_2$ does not contain a maximal monotone path after percolation.
However, we manage to replace the (dashed) edges $e_1$ and $e_3$ by paths of length $2k_1+1$ and $2k_2+1$ in $(R_1)_p$ and $(R_3)_p$, respectively.}
\label{fig:1}
\end{figure}

The next part of the proof consists of arguing about cycles of all lengths in $I_2$. To show this, we will utilise Lemma~\ref{l: very short cycles} in an inductive way. The following lemma will serve as the key inductive step.

\begin{lemma}\label{l: inductive step}
Let $D$ be a sufficiently large integer and $\rho\geq  2^7\e/D$. Let $H\subseteq Q^D$ be the subcube of dimension $D/2$ obtained by fixing the last $D/2$ coordinates to be $0$. 
Condition on the graph $H_\rho$ and the event that it contains a cycle $C$ of length $2L\in [D/100, D^{11}]$.
Then, with probability at least $1-\exp\{-\Omega(D)\}$, we have that $(Q^d\setminus E(H))_\rho\cup C$ contains simultaneously all cycles of even length in the interval $[2L+2^{-5}D,2^{-8} LD]$.
\end{lemma}
\begin{proof}
Fix $2\ell\in [2L+2^{-5}D,2^{-8}LD]$ for $\ell\in \mathbb N$.  Let $k$ be such that $2L+2k=2\ell$ and note that $2^{-6}D\le k\le 2^{-9}LD$. Note further that $k$ can be written as $k=t\cdot k_1+k_2$, where $k_1=2^{-5}D$, $t\le 2^{-4}L$, and $k_2\in [2^{-6}D,3\cdot 2^{-6}D]$ (in particular, $2\ell=2L+2tk_1+2k_2$).

Fix $L$ vertex-disjoint edges in $C$ and denote them by $e_1=u_1u_1',\ldots, e_L=u_Lu_L'$. For each $i\in [L]$, let $R_i$ be the subcube obtained by fixing the \textit{first} $D/2-1$ coordinates on which $\{u_i,u_i'\}$ agree, and letting the other $D/2+1$ coordinates vary. 
Note that, since $H$ has the \textit{last} $D/2$ coordinates fixed to be $0$ and $C\subseteq H$, for each $i\in [L]$, we have that $H\cap R_i=e_i$. Furthermore, for every $i\neq j\in [L]$, we have that $u_j,u_j'$ disagree on some of the first $D/2$ coordinates with each of $u_i,u_i'$, and thus $R_i\cap R_j=\varnothing$.

Now, for each $i\in [L]$, let $R_{u_i}\subseteq R_i$ be the subcube obtained by fixing the first $D/2$ coordinates to be as in $u_i$, and let $R_{u_i'}$ be the subcube obtained by fixing the first $D/2$ coordinates to be as in $u_i'$. 
Note that $V(R_i)=V(R_{u_i})\cup V(R_{u_i'})$, $R_{u_i}\cap R_{u_i'}=\varnothing$, and that each of $R_{u_i}$ and $R_{u_i'}$ shares no edges with $H$. For every $i\in [L]$, we orient the cubes $R_{u_i}$ and $R_{u_i'}$ such that $u_i,u_i'$ are their all-zero coordinate. 
Thus, the $j$-th layer in $R_{u_i}$ (resp.\ in $R_{u_i'}$) is the set of vertices in $R_{u_i}$ (resp.\ in $R_{u_i'}$) at distance $j$ from $u_i$ (resp.\ from $u_i'$). There are $\binom{D/2}{j}$ vertices in these layers. 

Then, for every $i\in [L/2]$, we expose the edges of $(Q^d\setminus E(H))_\rho$ between the $k_1$-th layer in $R_{u_i}$ and the $k_1$-th layer in $R_{u_i'}$. Since $D$ is sufficiently large, for each $i$, the probability that none of these edges are in $(Q^d\setminus E(H))_\rho$ is $(1-\rho)^{\binom{D/2}{k_1}}\le \frac{1}{10}$.

For each $i$ such that at least one such edge exists, we choose one of these edges arbitrarily and denote it by $f_i=w_iw_i'$ where $w_i\in V(R_{u_i}), w_i'\in V(R_{u_i'})$. 
For each such $i$, let $Q(u_i,w_i)\subseteq R_{u_i}$ be the unique subcube of dimension $k_1$ which contains $u_i$ and $w_i$. 
Note that there is a natural isomorphism from $Q(u_i,w_i)$ to $Q^{k_1}$ which maps $u_i$ to the all-zero vector; thus, when we talk about monotone paths in $Q(u_i,w_i)$, we mean a path whose image under this isomorphism is monotone. Define $Q(u_i',w_i')\subseteq R_{u_i'}$ in a similar manner. 

We note that $u_iu_i',w_iw_i'\in E((Q^d\setminus E(H))_\rho\cup C)$, and that all the edges of $Q(u_i,w_i),Q(u_i',w_i')$ have not been exposed so far.  Also, if there is a monotone path from $u_i$ to $w_i$ in $Q(u_i,w_i)_\rho$, and from $u_i'$ to $w_i'$ in $Q(u_i',w_i')_\rho$, then replacing $e_i$ with these two paths and $f_i$ increases the length of $C$ by $2k_1$ (see Figure \ref{fig:1}). Since $\rho k_1=\rho 2^{-5}D>e$, by Theorem \ref{thm:monotone}, the probability that $Q(u_i,w_i)_\rho$ contains a (maximal) monotone path from $u_i$ to $w_i$ is at least $1/2$, and the same holds for the probability of the existence of a maximal monotone path in $Q(u_i',w_i')_\rho$. 
Hence, the number of $i\in [L/2]$ for which both $f_i$ and the pair of maximal monotone paths in $Q(u_i,w_i)$ and $Q(u'_i,w'_i)$ are contained in $Q^D_\rho$ stochastically dominates $\mathrm{Bin}\left(L/2,\frac{9}{10}\cdot\frac{1}{4}\right)$: indeed, the edge $f_i$ is present in $(R_i)_\rho$ with probability at least $\frac{9}{10}$ and, conditionally on it, the said two monotone paths exist independently with probability at least $1/2$ each.
By Lemma \ref{l: chernoff} we have that, with probability $1-\exp\{-\Omega(D)\}$, there are at least $L/9$ such $i$. Choosing $t\le 2^{-4}L<L/9$ such $i$  allows us to increase the length of $C$ by $2tk_1$. 

We can now repeat the same argument for $i\in [L/2+1,L]$, this time seeking only one suitable $i$ for $k_2$ (instead of $t$ for $k_1$), which follows verbatim. This produces, with probability at least $1-\exp\{-\Omega(D)\}$, a cycle of length $2L+2tk_1+2k_2=2\ell$ (see Figure \ref{fig:1}). A union bound over the at most $LD\le D^{12}$ possible choices of $k$ completes the proof.
\end{proof}

With \Cref{l: inductive step} at hand, we are ready to show the typical existence of cycles of lengths in~$I_2$. 

\begin{proof}[Proof of \Cref{thm:main1} for the lengths in $I_2$]
For every $i \in \{0,\ldots, 10\}$, define $Q_{(i)}$ to be the subcube of $Q^d$ obtained by fixing the last $(1-2^{i-12})d$ coordinates to be zero, and letting the remaining $2^{i-12}d$ coordinates vary. Let $H_i = (Q_{(i)})_p$.

We claim that there exist constants $b_0, b_1,\ldots,b_{10} >0$ such that \textbf{whp} $H_i$ contains simultaneously every cycle of even length in the interval $[4,b_id^{i+1}]$. We prove the claim inductively. 
As a base case, since $c$ is sufficiently large, by Lemma \ref{l: very short cycles} applied to $H_0=(Q_{(0)})_p$ (that is, with $D=2^{-12} d$), \textbf{whp} $H_0$ contains simultaneously every cycle of even length in the interval $[4,2^{-15}d]$, and so we can take $b_0 = 2^{-15}$. 
Moreover, conditionally on the latter event, a direct application of \Cref{l: inductive step} shows the claim for $i=1$.

Now, fix $i \in [9]$ and suppose that there exists a constant $b_i>0$ such that \textbf{whp} $H_i$ contains simultaneously every cycle of even length in the interval $[4,b_id^{i+1}]$. We expose the edges of $H_i$ and condition on this event.

Fix a cycle $C_i$ of length $b_id^{i+1}-d/2$ in $H_i$. Then, applying Lemma \ref{l: inductive step} with $D = 2^{i-11}d$ to $H_i, C_i$ in the ambient cube $Q_{(i+1)}$, we see that there exists a constant $b_{i+1}>0$ such that \textbf{whp} $(H_{i+1} \setminus H_i) \cup C_i$ contains simultaneously every cycle of even length in the interval 
\[[b_id^{i+1}-d/2+2^{-5}\cdot 2^{i-11}d, 2^{-8}\cdot ((b_id^{i+1}-d/2)/2)\cdot (2^{i-11}d)] \supseteq [b_id^{i+1},b_{i+1}d^{i+2}],\] 
where we take $b_{i+1}=2^{i-21}b_i$. 
Hence, since $H_i \subseteq H_{i+1}$, \textbf{whp} $H_{i+1}$ contains simultaneously every cycle of even length in the interval $[4, b_id^{i+1}]\cup [b_id^{i+1},b_{i+1}d^{i+2}]=[4,b_{i+1}d^{i+2}]$. 

As a result, there exists a constant $b_{10}>0$ such that \textbf{whp} $H_{10} \subseteq Q^d_p$ contains simultaneously every cycle of even length in the interval $[4,b_{10}d^{11}]\supseteq I_2$, as required.
\end{proof}

\paragraph{Finding cycles with length in the middle range.}
We now turn to finding cycles with lengths in the interval $I_3$ defined in the beginning of Section~\ref{sec:proof}. We will build upon the cycle whose typical existence is guaranteed by Theorem~\ref{thm:long cycle} with $q=1$ and $c$ sufficiently large.

Formally, let us partition $Q^d$ into four disjoint subcubes $(Q_{i,j})_{i,j\in \{0,1\}}$ according to the values of the first two coordinates.
Note that, by Theorem~\ref{thm:long cycle} applied for dimension $d-2$, \textbf{whp} there exists a cycle of length at least $2^{d-3}$ in $(Q_{0,0})_p$.
In the sequel, we condition on a cycle $C \subseteq Q_{0,0}$ of length at least $2^{d-3}$.

For every $\bm{j}\in\{0,1\}^{d/2-2}$, let $Q_{(0,0)}(\bm{j})$ be the subcube obtained by fixing the first two coordinates to be zero, the vector consisting of the next $d/2-2$ coordinates to be $\bm{j}$, and letting the last $d/2$ coordinates vary. 
Note that $\cQ_0\coloneqq\{Q_{(0,0)}(\bm{j})\colon \bm{j}\in \{0,1\}^{d/2-2}\}$ is a partition of $Q_{0,0}$ into $2^{d/2-2}$ disjoint hypercubes of dimension $d/2$.
Similarly, for every $\bm{j}\in\{0,1\}^{d/2-2}$, we define $Q_{1,0}(\bm{j})$ to be the subcube obtained by fixing the first two coordinates to be $(1,0)$, the vector consisting of the next $d/2-2$ coordinates to be $\bm{j}$, and letting the last $d/2$ coordinates vary
In particular, we have that $\cQ_1\coloneqq \{Q_{1,0}(\bm{j})\colon \bm{j}\in \{0,1\}^{d/2-2}\}$ is a partition of $Q_{1,0}$ into $2^{d/2-2}$ disjoint hypercubes of dimension $d/2$.

For any length $2\ell\in I_3$, by a straightforward greedy approach, one can find (possibly intersecting) paths $P_1(\ell),\ldots,P_{d^4}(\ell)$ with the following properties:
\begin{enumerate}[\upshape{(P\arabic*)}]
\itemsep -0.5mm
    \item\label{i:paths} each of these paths is contained in $C$,
    \item the length of each path is equal to $2\ell - d^2$,
    \item\label{i:distinct} for every distinct $i,j\in [d^4]$, every endpoint $u$ of $P_i(\ell)$ and every endpoint $v$ of $P_j(\ell)$, $u$ and $v$ are in distinct subcubes in $\cQ_0$,
\end{enumerate}
(indeed, note that there are $\Omega(2^d)$ paths of length $2\ell-d^2$ contained in $C$, and there are $2^{d/2}$ vertices in each subcube in $\cQ_0$).
Given $2\ell\in I_3$ and $i \in [d^4]$, we denote by $u_{i,\ell}$ and $v_{i,\ell}$ the two endpoints of $P_i(\ell)$, and by $u'_{i,\ell}$ and $v'_{i,\ell}$ the neighbours in $Q_{1,0}$ of $u_{i,\ell}$ and $v_{i,\ell}$, respectively. 
Let $\bm{j}(u_{i,\ell}),\bm{j}(v_{i,\ell})$ be the coordinates such that $u_{i,\ell}\in Q_{0,0}(\bm{j}(u_{i,\ell})),v_{i,\ell}\in Q_{0,0}(\bm{j}(v_{i,\ell}))$. 
Then, by construction of $\cQ_0,\cQ_1$, we have that $u_{i,\ell}'\in Q_{1,0}(\bm{j}(u_{i,\ell}))\eqqcolon Q(u'_{i,\ell})$ and $v_{i,\ell}'\in Q_{1,0}(\bm{j}(v_{i,\ell}))\eqqcolon Q(v_{i,\ell}')$. In particular, property~\ref{i:distinct} implies that, for distinct $i,j \in [d^4]$, the two sets $\{ Q(u'_{i,\ell}), Q(v'_{i,\ell})\}$ and $\{Q(u'_{j,\ell}),Q(v'_{j,\ell})\}$ are disjoint. However, note that either set may nevertheless consist of a single subcube.

We require the following lemma.
\begin{lemma}\label{lem:local}
Let $C \subseteq Q_{0,0}$ be a cycle of length at least $2^{d-3}$, $c>0$ be  sufficiently large  and $p = c/d$. 
For each  $2\ell \in I_3$ let $P_1(\ell),\ldots, P_{d^4}(\ell)$ be a family of paths satisfying properties \emph{\ref{i:paths}--\ref{i:distinct}}. Then, \textbf{whp}, for every $2\ell\in I_3$, there is an index $i\in [d^4]$ such that each of the following properties holds simultaneously:
\begin{enumerate}[(i)]
\itemsep -0.5mm
    \item each of the edges $u_{i,\ell}u'_{i,\ell}$ and $v_{i,\ell}v'_{i,\ell}$ belongs to $Q^d_p$,\label{i: first point}
    \item the vertices $u'_{i,\ell},v'_{i,\ell}$ belong to components of order at least $2^{d/2}/d$ in $Q(u'_{i,\ell})_p$ and $Q(v'_{i,\ell})_p$, respectively,\label{i: second point}
    \item the shortest path $P'_{i,\ell}$ between $v'_{i,\ell}$ and $u'_{i,\ell}$ in $(Q_{1,0})_p$ has length at most $d(\log d)^3$.\label{i: third point}
\end{enumerate}
\end{lemma}
\begin{proof}
Fix a length $2\ell\in I_3$. Then, by Theorem \ref{th: bond percolation}, for every $i\in [d^4]$, the probability that each of \ref{i: first point} and \ref{i: second point} holds for $P_i(\ell)$ is at least $p^2\cdot (1-o(1))y^2\ge p^2y^2/2$, where $y=y(c/2)>0$ (as the dimension of the relevant cubes is $d/2$) is the constant guaranteed by Theorem \ref{th: bond percolation}\ref{i: bond giant}. Indeed, if $Q(u'_{i,\ell}) \neq Q(v'_{i,\ell})$, then this is clear by independence and, if $Q(u'_{i,\ell})=Q(v'_{i,\ell})$, then the claim holds by Harris' inequality (see, e.g.,~\cite[Theorem 6.3.3]{AS16}) and the observation that the property of belonging to a component of order at least $2^{d/2}/d$ in $Q(u'_{i,\ell})_p= Q(v'_{i,\ell})_p$ is increasing.
Moreover, by property \ref{i:distinct}, these events are jointly independent for different $i\in [d^4]$. 
Therefore, we have that, with probability at least $1-\mathbb P[\mathrm{Bin}(d^4, p^2y^2/2) = 0] = 1-o(2^{-d})$, there is some $i \in [d^4]$ such that \ref{i: first point}-\ref{i: second point} hold. In particular, by a union bound, this is true for all $2\ell \in I_3$.

Finally, we note that, by Theorem \ref{th: bond percolation}, \textbf{whp} $(Q_{1,0})_p$ contains a unique component $L$ of order at least $2^d/d$, and all the other components have order $O(d)$. Hence every vertex that lies in a component of size at least $2^{d/2}/d$ belongs to $L$. In addition, by Theorem \ref{th: bond percolation}, $L$ has diameter at most $d(\log d)^3$. 
In particular, \textbf{whp}, for any $i \in [d^4]$ such that \ref{i: first point}-\ref{i: second point} hold, $v'_{i,\ell}$ and $u'_{i,\ell}$ lie in $L$
and are thus joined by a path  $P'_{i,\ell}$ of length at most $d(\log d)^3$ in $(Q_{1,0})_p$. 
\end{proof}

Note that, if the conclusion of \Cref{lem:local} holds for some $2\ell \in I_3$ and $i \in [d^4]$, then there is a cycle $C'=u_{i,\ell}P_i(\ell)v_{i,\ell}v'_{i,\ell}P'_{i,\ell}u'_{i,\ell}u_{i,\ell}$ in $(Q_{0,0} \cup Q_{1,0})_p$ of length at most $2\ell - d^2 + d(\log d)^3 + 2$. Furthermore, $C' \cap Q_{0,0} = P_i(\ell)$.

With \Cref{lem:local} at hand, we are now ready to complete the proof of \Cref{thm:main1} for cycles of lengths in $I_3$.

\begin{proof}[Proof of \Cref{thm:main1} for the lengths in $I_3$]
First, we expose the edges of $(Q_{0,0}\cup Q_{1,0})_p$. 
By Theorem~\ref{thm:long cycle}, \textbf{whp} there exists a cycle $C$ of length at least $2^{d-3}$ in $(Q_{0,0})_p$. 
For each $2\ell \in I_3$, let $P_1(\ell),\ldots, P_{d^4}(\ell)$ be a family of paths satisfying properties \ref{i:paths}-\ref{i:distinct}. 
Then, by \Cref{lem:local}, \textbf{whp}, for each $2\ell \in I_3$, we have that $(Q_{0,0}\cup Q_{1,0})_p$ contains a cycle $C'(\ell)$ of length $|C'(\ell)|\le 2\ell-d^2+d(\log d)^3 +2$ such that $C'(\ell) \cap Q_{0,0}$ is a path of length $2\ell-d^2$.

For each $2\ell \in I_3$, we fix vertex-disjoint edges $e_1(\ell),\ldots,e_{\ell/3}(\ell)$ in $C'(\ell)\cap Q_{0,0}$. 
For each such edge $e_i(\ell)$, there is a unique path $T_i(\ell)$ of length 3 with the same endpoints as $e_i(\ell)$ whose middle edge lies in the hypercube $Q_{0,1}$.
Note that, by construction, the paths $T_1(\ell),\ldots,T_{\ell/3}(\ell)$ are disjoint from each other and lie outside the set of exposed edges $E(Q_{0,0}\cup Q_{1,0})$, so the probability that any of them is entirely included in $Q^d_p$ is $p^3$.
Hence, a standard application of Chernoff's bound (\Cref{l: chernoff}) together with a union bound over the at most $2^d$ possible values of $2\ell \in I_3$ implies that, \textbf{whp}, for every $2\ell\in I_3$, at least $d^{-4}\ell\ge d^2$ of the paths $T_1(\ell),\ldots,T_{\ell/3}(\ell)$ are contained in $Q^d_p$. 
By replacing $(2\ell - |C'(\ell)|)/2\le d^2$ edges $e_i(\ell)$ with $ i \in [\ell/3]$ with the corresponding paths $T_i(\ell)$, this produces a cycle of length exactly $2\ell$, as desired.
\end{proof}

\paragraph{Finding long cycles.} 
Finally, it remains to ensure that cycles of lengths in the interval $I_4$ (defined in the beginning of Section~\ref{sec:proof}) exist in $Q^d_p$. 
Given $\delta \in[0,1]$, partition $V(Q^d)$ into sets $V_1=V_1(\delta), V_2=V_2(\delta)$ and $V_3=V_3(\delta)$ where a vertex is assigned to $V_1$ with probability $q_1=1-\delta/2$, and to any of $V_2,V_3$ with probability $q_2 = q_3 = \delta/4$, independently for different vertices.
We will make use of the following lemma.
\begin{lemma}\label{l: typical properties}
For every $\eps\in (0,1)$, there exist constants $c(\eps)>0$ and $\delta\in[0,1]$  such that, with $V_1,V_2,V_3$ as above, for all $c\geq c(\eps)$ with $p = p(d) = c/d$, each of the following holds \textbf{whp}.  
\begin{enumerate}[\upshape{(P\arabic*')}]
    \item\label{prop:A3} $Q^d_p[V_1]$ contains a cycle $C_1$ of length at least $(1-\eps)2^d$.
    \item\label{prop:A4} $Q^d_p[V_2]$ contains a unique component $L_2$ of size at least $2^d/d$ whose diameter is at most $2^d/d^8$.
    \item\label{prop:A5} All but $o(2^d/d^{10})$ vertices of $Q^d$ have at least one neighbour (in $Q^d$) in $L_2$.
    \item\label{prop:A6} For all but $o(2^d/d^{10})$ edges $uv\in E(Q^d)$, there are vertices $u',v'\in V_3$ such that $u,v,v',u'$ (in this order) form a $4$-cycle in $Q^d$.
\end{enumerate}
\end{lemma}
\begin{proof}
Let $\delta\in (0,1)$ be such that $(1-\delta)(1-\delta/2)=1-\eps$. \Cref{thm:long cycle}, with $q=1-\delta$, implies that \textbf{whp} $Q^d_p[V_1]$ contains a cycle of length $(1-\delta)(1-\delta/2)2^d=(1-\eps)2^d$ and hence \ref{prop:A3} holds.

\ref{prop:A4} follows from Theorem \ref{th: mixed percolation}, with $L_2$ being the unique giant component of $Q^d_p[V_2]$.

For \ref{prop:A5}, fix a vertex $v\in V(Q^d)$ and some $\delta d$ of its neighbours in $Q^d$, denoted $w_1,\ldots, w_{\delta d}$. Consider $\delta d$ vertex-disjoint subcubes $Q(1),\ldots, Q(\delta d)$, each of dimension at least $(1-\delta)d$ and such that, for all $i\in [\delta d]$, $Q(i)$ contains $w_i$ and no $w_j$ for $j \in [\delta d]\setminus \{i\}$ (such exist by, for example, Claim 2.2 in~\cite{DK22}). 
Then, by symmetry of the hypercube and by Theorem~\ref{th: mixed percolation}, there exists a constant $\alpha>0$ such that, for every $i\in[\delta d]$, with probability at least $\alpha$, $w_i$ belongs to a giant component in $(Q(i))_p[V_2]$ (which is, in fact, \textbf{whp} part of the giant component of $Q^d_p[V_2]$ by the discrepancy in the sizes of the components ensured in Theorem~\ref{th: mixed percolation}\ref{i: mixed giant}). Moreover, since the subcubes $Q(1),\ldots,Q(\delta d)$ are vertex-disjoint, these events are independent. Thus, assuming Theorem~\ref{th: mixed percolation}\ref{i: mixed giant}, the probability that none of the neighbours of $v$ belongs to the giant component of $Q^d_p[V_2]$ is at most $(1-\alpha)^{\delta d}\le \exp\{-\alpha\delta d\}$. 
Thus, by Markov's inequality, all but $o(2^d/d^{10})$ vertices in $Q^d$ have at least one neighbour in the giant component of $Q^d_p[V_2]$ (where we note that the exponent $10$ in the power of $d$ is rather arbitrary).

For \ref{prop:A6}, fix an edge $uv\in E(Q^d)$. There are $d-1$ distinct pairs of vertices $u',v'$ such that $uvv'u'$ form a 4-cycle in $Q^d$.
The probability that none of these pairs belongs to $V_3$ is $(1-q_3^2)^{d-1}\le \exp\{-\delta^2 d/20\}$. 
Thus, the expected number of edges $uv\in E(Q^d)$ such that there are \textit{no} vertices $u',v'\in V_3$ with $uvv'u'$ forming a 4-cycle is at most $\exp\{-\delta^2 d/20\} d2^d$. 
Thus, by Markov's inequality, \textbf{whp} the latter event holds for $o(2^d/d^{10})$ pairs only, as desired.
\end{proof}

We are ready to complete the proof of \Cref{thm:main1}.
\begin{proof}[Proof of \Cref{thm:main1} for the lengths in $I_4$]
Let $V_1,V_2,V_3$ be as in \Cref{l: typical properties} and $c>0, \delta\in[0,1]$ be such that \ref{prop:A3}-\ref{prop:A6} hold \textbf{whp}. 
We begin by exposing $V_1,V_2,V_3$ and $Q^d_p[V_1] \cup Q^d_p[V_2]$. 
We condition on \ref{prop:A3}-\ref{prop:A6}. We denote by $C$ a cycle of length at least $(1-\eps)2^d$ in $Q_p^d[V_1]$, given by \ref{prop:A3}, and by $L_2$ the giant component of $Q^d_p[V_2]$, given by \ref{prop:A4}.

Fix $2\ell\in [2^{d-4}, |C|]$ and $k=2^d/d^8$. We enumerate $C=\{v_1,\ldots, v_{|C|}\}$ and let $S_1=\{v_1,\ldots, v_k\}$ and $S_2=\{v_{2\ell-2k+1},\ldots, v_{2\ell-k}\}$. Note that these sets are well-defined and disjoint. 

By \ref{prop:A5}, the number of vertices $u\in S_1$ with a neighbour $v$ in the giant component of $Q^d_p[V_2]$ is at least $k/2$. 
Next, expose the edges of $Q^d_p$ between $V_1$ and $V_2$. Then, the probability that there are no edges in $Q^d_p$ between $S_1$ and $L_2$ is at most $(1-p)^{k/2}\le \exp\left\{-2^{d-1}/d^9\right\}$. 
A similar argument shows that the probability that there are no edges in $Q^d_p$ between $S_2$ and $L_2$ is also at most $\exp\left\{-2^{d-1}/d^9\right\}$.

Hence, with probability at least $1-2\exp\left\{-2^{d-1}/d^9\right\}$, there exist $u_1\in S_1, u_2\in S_2$, and $w_1,w_2 \in L_2$, such that $u_1w_1,u_2w_2\in E(Q^d_p)$. By \ref{prop:A4}, there is a path $P \subseteq Q^d_p[V_2]$ of length at most $k$ between $w_1$ and $w_2$. Denote by $P'$ the path between $u_1$ and $u_2$ along the cycle $C$ containing $v_{k+1}$. Then, with probability at least $1-2\exp\{-2^d/d^9\}$, there is a cycle $C'\coloneqq P' \cup P \cup \{u_1w_1,u_2w_2\}$ whose length is between $2\ell-3k+2$ and $(2\ell-k-1)+k+2=2\ell+1$ (so at most $2\ell$ by parity considerations; see Figure \ref{fig:2}).

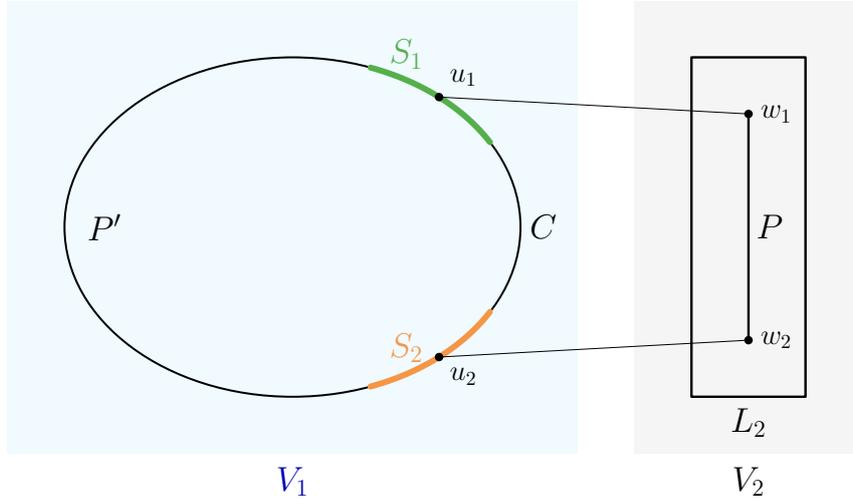
\begin{figure}[H]
\centering
\scalebox{0.75}
{
\begin{tikzpicture}[font=\sffamily, line cap=round, line join=round]

  %%% Styles & Colors %%%
  \definecolor{bgA}{RGB}{240,250,255}
  \definecolor{bgB}{RGB}{245,245,245}
  \definecolor{S1col}{RGB}{85,175,75}    % green
  \definecolor{S2col}{RGB}{245,150, 70}  % orange
  \tikzset{
    bigcycle/.style={draw=black, line width=1pt},
    arc/.style={line width=3pt},         % heavier than circle
    connector/.style={thin},
    nodept/.style={circle,fill=black,inner sep=1.5pt},
    labelsmall/.style={font=\footnotesize\itshape},
  }

  %%% Backgrounds %%%
  \fill[bgA]  (-5,-4) rectangle (5,4);
  \fill[bgB]  ( 6,-4) rectangle (10,4);

  %%% Left region: the cycle %%%
  \draw[bigcycle] (0,0) ellipse (4cm and 3cm);

  % highlighted arcs
  \draw[arc, draw=S1col] plot [domain=30:70, samples=50]
    ({4*cos(\x)}, {3*sin(\x)});
  \draw[arc, draw=S2col] plot [domain=-70:-30, samples=50]
    ({4*cos(\x)}, {3*sin(\x)});

  % key points on ellipse
  \coordinate (u1) at ({4*cos(50)}, {3*sin(50)});
  \coordinate (u2) at ({4*cos(-50)}, {3*sin(-50)});
  \node[nodept,label=above right:\Large{$u_1$}] at (u1) {};
  \node[nodept,label=below right:\Large{$u_2$}] at (u2) {};

  % place S1,P1 and S2,P2 labels
  \coordinate (mS1) at ({4*cos(60)}, {3*sin(60)});
  \coordinate (mP1) at ({4*cos(43)}, {3*sin(43)});
  \coordinate (mS2) at ({4*cos(-60)}, {3*sin(-60)});
  \coordinate (mP2) at ({4*cos(-36)}, {3*sin(-36)});

  \node[labelsmall,text=S1col, above=2pt]  at (mS1) {\LARGE{$S_1$}};
  %\node[labelsmall,text=S1col, below=2pt]  at (mP1) {$P_1$};
  \node[labelsmall,text=S2col, above=2pt]  at (mS2) {\LARGE{$S_2$}};
  %\node[labelsmall,text=S2col, below=2pt]  at (mP2) {$P_2$};

  % central labels
  \node at (-3.3,0) {\LARGE{$P'$}};
  \node[labelsmall] at (4.4,0) {\LARGE{$C$}};

  % V1 label
  \node[blue!70!black, font=\bfseries] at (0,-4.5) {\LARGE{$V_1$}};

  %%% Right region: the gadget %%%
  \begin{scope}[xshift=8cm]
    % outer rectangle L2
    \draw[very thick] (-1,3) rectangle (1,-3);

    % inner vertical path P
    \coordinate (w1) at (0,2);
    \coordinate (w2) at (0,-2);
    \draw[very thick] (w1) -- (w2);
    \node[nodept,label=right:\Large{$w_1$}] at (w1) {};
    \node[nodept,label=right:\Large{$w_2$}] at (w2) {};

    % label P to the right of the middle
    \node[labelsmall, right=3pt] at (-0.1,0) {\LARGE{$P$}};

    % label L2 beneath
    \node[labelsmall, below=2pt] at (0,-3) {\LARGE{$L_2$}};
    % V2 label
    \node[font=\bfseries] at (0,-4.5) {\LARGE{$V_2$}};
  \end{scope}

  %%% Connecting lines %%%
  \draw[connector] (u1) -- (8,2);
  \draw[connector] (u2) -- (8,-2);

\end{tikzpicture}
}
\caption{The construction of the cycle $C'$ using the paths $P'\subseteq C$, $P\subseteq L_2$ and the edges $u_1w_1,u_2 w_2$ connecting the endpoints of $P'$ and $P$.}
\label{fig:2}
\end{figure}

By   \ref{prop:A6}, there are at least $\frac{2\ell- 3k + 2}{2} - o(2^d/d^{10})\ge \ell/2$ vertex-disjoint edges $uv$ in $P'$, such that there are vertices $u',v'\in V_3$ such that $uvv'u'$ is a 4-cycle in $Q^d$. 
Since every vertex in $V_3$ can participate in at most $\tbinom{d}{2}\le d^2/2$ such cycles, there are at least $\ell/(2d^2)$ vertex-disjoint $4$-cycles $uvv'u'$ of this form in $Q^d$. 

Finally, we expose the edges in $Q^d_p$ between $V_1$ and $V_3$ and inside $V_3$. 
Then, each of the said $4$-cycles belongs to $Q^d_p$ with probability $p^3$.
Hence, by Chernoff's bound (Lemma \ref{l: chernoff}), the probability that there are less than $2^d/d^6$ vertex-disjoint edges $uv$ in $P'$ such that there are $u',v'\in V_3$ with $uvv'u'$ forming a 4-cycle in $Q^d_p$ is at most $\exp\{-2^d/d^6\}$. 
Replacing $(2\ell-|C'|)/2\le 3k=o(2^d/d^6)$ of these edges with the vertex-disjoint $3$-paths going through $V_3$, we obtain that, with probability at least 
\[1-2\exp\{-2^{d-1}/d^9\}-2\exp\{-2^d/d^6\} = 1-o(2^{-d}),\]
a cycle of length exactly $2\ell$ exists. 
A union bound over the at most $2^d$ choices of $\ell$ completes the proof.
\end{proof}

\section*{Acknowledgement}
This research was funded in part by the Austrian Science Fund (FWF) [10.55776/P36131 (J. Erde), 10.55776/F1002 (M. Kang), 10.55776/ESP624 (L. Lichev), 10.55776/ESP3863424 (M. Anastos)], and by NSF-BSF grant 2023688 (M. Krivelevich). For open access purposes, the authors have applied a CC BY public copyright license to any author accepted manuscript version arising from this submission.

\bibliographystyle{abbrv}
\bibliography{Bib}

\end{document}